\documentclass{amsart}
\usepackage{amssymb, amsmath}
\usepackage[a4paper,top=2cm,bottom=2cm,left=1.5cm,
right=1.5cm]{geometry}
\usepackage{array}
\usepackage{xcolor}

\newtheorem{Theorem}{Theorem}
\newtheorem{Definition}[Theorem]{Definition}
\newtheorem{Remark}[Theorem]{Remark}

\newtheorem{Lemma}[Theorem]{Lemma}
\newtheorem{Corollary}[Theorem]{Corollary}

\newcommand{\spa}{\mbox{span}}
\newcommand{\F}{F\langle X\rangle}

\newcommand{\Ft}{F\langle X, \mbox{Tr}\rangle}
\newcommand{\V}{\mbox{var}}
\DeclareMathOperator{\Id}{Id}
\newcommand{\I}{\mbox{Id}}

\newcommand{\tr}{\mbox{tr}}
\newcommand{\Tr}{\mbox{Tr}}

\begin{document}
\title[Trace identities and almost polynomial growth]{Trace identities and almost polynomial growth}

\author{Antonio Ioppolo}
\address{IMECC, UNICAMP, S\'ergio Buarque de Holanda 651, 13083-859 Campinas, SP, Brazil}
\email{ioppolo@unicamp.br}

\thanks{A. Ioppolo was supported by the Fapesp post-doctoral grant number 2018/17464-3}

\author{Plamen Koshlukov}
\address{IMECC, UNICAMP, S\'ergio Buarque de Holanda 651, 13083-859 Campinas, SP, Brazil}
\email{plamen@unicamp.br}

\thanks{P. Koshlukov was partially supported by CNPq grant No.~302238/2019-0, and by FAPESP grant No.~2018/23690-6}

\author{Daniela La Mattina}
\address{Dipartimento di Matematica e Informatica, Universit\`a degli Studi di Palermo, Via Archirafi 34, 90123, Palermo, Italy}
\email{daniela.lamattina@unipa.it}

\thanks{D. La Mattina was partially supported by GNSAGA-INDAM}



\begin{abstract}
In this paper we study algebras with trace and their trace polynomial identities over a field of characteristic 0. We consider two commutative matrix algebras: $D_2$, the algebra of $2\times 2$ diagonal matrices and $C_2$, the algebra of $2 \times 2$ matrices generated by $e_{11}+e_{22}$ and $e_{12}$. We describe all possible traces on these algebras and we study the corresponding trace codimensions. Moreover we characterize the varieties with trace of polynomial growth generated by a finite dimensional algebra. As a consequence, we see that the growth of a variety with trace is either polynomial or exponential. 
\end{abstract}

\maketitle

\section{Introduction}
All algebras we consider in this paper will be associative and over a fixed field $F$ of characteristic 0. Let $F\langle X\rangle$ be the free associative algebra freely generated by the infinite countable set $X=\{x_1, x_2, \ldots\}$ over $F$. One interprets $F\langle X\rangle$ as the $F$-vector space with a basis consisting of 1 and all non-commutative monomials (that is words) on the alphabet $X$. The multiplication in $F\langle X\rangle$ is defined on the monomials by juxtaposition. Let $A$ be an algebra, it is clear that every function $\varphi\colon X\to A$ can be extended in a unique way to a homomorphism (denoted by the same letter) $\varphi\colon F\langle X\rangle\to A$. A polynomial $f\in F\langle X\rangle$ is a polynomial identity (PI for short) for the algebra $A$ whenever $f$ lies in the kernels of all homomorphisms from $F\langle X\rangle$ to $A$. Equivalently $f(x_1,\ldots,x_n)$ is a polynomial identity for $A$ whenever $f(a_1,\ldots,a_n)=0$ for any choice of $a_i \in A$. The set of all PI's for a given algebra $A$ forms an ideal in $F\langle X\rangle$ denoted by $\I(A)$ and called the T-ideal of $A$. Clearly $\I(A)$ is closed under endomorphisms. It is not difficult to prove that the converse is also true: if an ideal $I$ in $F\langle X\rangle$ is closed under endomorphisms then $I=\I(A)$ for some (adequate) algebra $A$. One such algebra is the relatively free algebra $F\langle X\rangle/I$.

Knowing the polynomial identities satisfied by an algebra $A$ is an important problem in Ring theory. It is also a very difficult one; it was solved completely in very few cases. These include the algebras $F$ (trivial); $M_2(F)$, the full matrix algebra of order 2; $E$, the infinite dimensional Grassmann algebra; $E\otimes E$. If one adds to the above algebras the upper triangular matrices $UT_n(F)$, one will get more or less the complete list of algebras whose identities are known. 

The theory developed by A. Kemer in the 80-ies (see \cite{Kemer1991book}) solved in the affirmative the long-standing Specht problem: is every T-ideal in the free associative algebra finitely generated as a T-ideal? But the proof given by Kemer is not constructive; it suffices to mention that even the description of the generators of the T-ideal of $M_3(F)$ are not known, and it seems to be out of reach with the methods in use nowadays. 

Thus finding the exact form of the polynomial identities satisfied by a given algebra is practically impossible in the vast majority of important algebras. Hence one is led to study either other types of polynomial identities or other characteristics of the T-ideals. In the former direction it is worth mentioning the study of polynomial identities in algebras graded by a group or a semigroup, in algebras with involution, in algebras with trace and so on. Clearly one has to incorporate the additional structure into the ``new'' polynomial identities. It turned out such identities are sort of ``easier'' to study than the ordinary ones. 

We cite as an example the graded identities for the matrix algebras $M_n(F)$ for the natural gradings by the cyclic groups $\mathbb{Z}_n$ and $\mathbb{Z}$: these were described by Vasilovsky in \cite{Vasilovsky1998, Vasilovsky1999}. Gradings on important algebras (for the PI-theory) and the corresponding graded identities have been studied by very many authors (we refer the reader to the monograph \cite{ElduqueKochetov2013} and the references therein). The trace identities for the full matrix algebras were described independently by Razmyslov \cite{Razmyslov1974} and by Procesi \cite{Procesi1976} (see also the paper by Razmyslov \cite{Razmyslov1985} for a generalization to another important class of algebras). It turns out that the ideal of all trace identities for the matrix algebra $M_n(F)$ is generated by a single polynomial, this is the well known Hamilton--Cayley polynomial written in terms of the traces of the matrix and its powers (and then linearised). We must note here that as it  often happens, the simplicity of the statement of the theorem due to Razmyslov and Procesi is largely misleading, and that the proofs are quite sophisticated and extensive. 

The free associative algebra is graded by the degrees of its monomials, and also by their multidegrees. Clearly the T-ideals are homogeneous in such gradings; this implies that the relatively free algebras inherit the gradings on $F\langle X\rangle$. One might want to describe the Hilbert (or Poincar\'e) series of the relatively free algebras. This task was achieved also in very few instances. 

Studying the relatively free algebras is related to studying varieties of algebras. We recall the corresponding notions and their importance. Let $A$ be an algebra with T-ideal $\I(A)$. The class of all algebras satisfying all polynomial identities from $\I(A)$ (and possibly some more PI's) is called the variety of algebras $\V(A)$ generated by $A$. Then the relatively free algebra $F\langle X\rangle/\I(A)$ is the relatively free algebra in $\V(A)$ (clearly there might be several algebras that satisfy the same polynomial identities as $A$, and passing to $\V(A)$ one may ``forget'' about $A$ and even look for some ``better'' algebra generating the same variety).

One of the most important numerical invariants of a variety (or a T-ideal) is its codimension sequence. Let $P_n$ denote the vector space in $F\langle X\rangle$ consisting of all multilinear polynomials in $x_1$, \dots, $x_n$. One may view $P_n$ as the vector subspace of $F\langle X\rangle$ with a basis consisting of all monomials $x_{\sigma(1)} x_{\sigma(2)}\cdots x_{\sigma(n)}$ where $\sigma\in S_n$, the symmetric group on $\{1,2,\ldots,n\}$. It is well known that whenever the base field is of characteristic 0, each T-ideal $I$ is generated by its multilinear elements, that is by all intersections $I\cap P_n$, $n\ge 1$. 

On the other hand $P_n$ is a left module over $S_n$ and it is clear that $P_n\cong FS_n$, the regular $S_n$-module. The T-ideals are invariant under permutations of the variables hence $I\cap P_n$ is a submodule of $P_n$. In this way one can employ the well developed theory of the representations of the symmetric (and general linear) groups and study the polynomial identities satisfied by an algebra. This approach might have seemed rather promising but in 1972 A. Regev \cite{Regev1972} established a fundamental result showing that if $I\ne 0$ then the intersections $I\cap P_n$ tend to become very large when $n\to\infty$. More precisely let $A$ be a PI-algebra (i.e., $A$ satisfies a non-trivial identity), and let $I=\I(A)$, denote by $P_n(A) = P_n/(P_n\cap I)$. Then $P_n(A)$ is also an $S_n$-module, its dimension $c_n(A)=\dim P_n(A)$ is called the $n$-th codimension of $A$ (or of the variety $\V(A)$, or of the relatively free algebra in $\V(A)$). Regev's theorem states that if $A$ satisfies an identity of degree $d$ then $c_n(A)\le (d-1)^{2n}$. Since $\dim P_n=n!$ this gives a more precise meaning of the above statement about the size of $P_n\cap I$. Recall that this exponential bound for the codimensions allowed Regev to prove that if $A$ and $B$ are both PI-algebras then their tensor product $A\otimes B$ is also a PI-algebra. 

But computing the exact values of the codimensions of a given algebra is also a very difficult task, and $c_n(A)$ is known for very few algebras $A$. This is exactly the same list as above, namely that of the algebras whose identities are known. Hence one is led to study the growth of the codimension sequences. In the eighties Amitsur conjectured that for each PI-algebra $A$, the sequence $(c_n(A))^{1/n}$ converges when $n\to\infty$ and moreover its limit is an integer. This conjecture was dealt with by Giambruno and Zaicev \cite{GiambrunoZaicev1998, GiambrunoZaicev1999} (see also \cite{GiambrunoZaicev2005book}): they answered in the affirmative Amitsur's conjecture. The above limit is called the PI-exponent of a PI-algebra, $\exp(A)$. Giambruno and Zaicev's important result initiated an extensive research concerning the asymptotic behaviour of the codimension sequences of algebras. It is well known (see for example \cite[Chapter 7]{GiambrunoZaicev2005book} and the references therein) that either the codimensions of $A$ are bounded by a polynomial function or grow exponentially. The results of Giambruno and Zaicev also hold for the case of graded algebras (\cite{AljadeffGiambruno2013, AljadeffGiambrunoLaMattina2011, GiambrunoLaMattina2010}), algebras with involution (\cite{GiambrunoMiliesValenti2017}), superalgebras with superinvolution, graded involution or pseudoinvolution (\cite{Ioppolo2018, Ioppolo2020, Santos2017}) and also for large classes of non-associative algebras. It is useful to highlight that there are examples of non-associative algebras such that their PI-exponent exists but is not an integer and also examples where the PI-exponent does not exist at all. 

In this paper we study trace polynomial identities. We focus our attention on two commutative subalgebras of $UT_2$, the algebra of $2 \times 2$ upper-triangular matrices over $F$: $D_2$ the algebra of diagonal matrices and $C_2 = \mbox{span}_F \{e_{11}+ e_{22}, e_{12} \}$. In \cite[Theorem 2.1]{Berele1996} A. Berele described the ideal of trace identities for the algebra $D_n$; he proved that it is generated by the commutativity law and by the Hamilton--Cayley polynomial. The asymptotic behaviour of the codimensions of trace identities was studied by A. Regev. In fact he described in \cite{Regev1988} the asymptotics of the ordinary codimensions of the full matrix algebra, and in \cite{Regev1984} he proved that the ordinary and the trace codimensions of the full matrix algebra are asymptotically equal. 

Our main goal in this paper is the description of the varieties of trace algebras that are of \textsl{almost polynomial growth}. This means that the codimensions of the given variety are of exponential growth but each proper subvariety is of polynomial growth. The description we obtain is in terms of excluding the algebras $D_2$ and $C_2$ with non-zero traces and the algebra $UT_2$ with the zero trace. As a by-product of the proof we obtain that the codimension growth of the trace identities of a finite dimensional algebra is either polynomially or exponentially bounded. 

It is interesting to note that if one considers the algebra $D_n$ of the diagonal $n\times n$ matrices without a trace, it is commutative and non-nilpotent, and hence its codimensions are equal to 1. But when adding a trace then it becomes of exponential growth. Hence in the case of diagonal matrices there cannot be a direct analogue of Regev's theorem mentioned above.

We recall here that similar descriptions for the ordinary codimensions can be found in \cite[Theorem 7.2.7]{GiambrunoZaicev2005book} where it was proved that the only two varieties of (ordinary) algebras of almost polynomial growth are the ones generated by the Grassmann algebra $E$ and by $UT_2$. In the case of algebras with involution, superinvolution, pseudoinvolution or graded by a finite group, a complete list of varieties of algebras of almost polynomial growth  was exihibited in \cite{GiambrunoIoppoloLaMattina2016, GiambrunoIoppoloLaMattina2019, GiambrunoMishchenko2001bis, GiambrunoMishchenkoZaicev2001, IoppoloLaMattina2017, IoppoloMartino2018, LaMattina2015, Valenti2011}.

In order to obtain our results we use methods from the theory of trace polynomial identities together with a version of the Wedderburn--Malcev theorem for finite dimensional trace algebras. Here we recall that a trace function on the matrix algebra $M_n(F)$ is just a scalar multiple of the usual matrix trace. In sharp contrast with this there are very many traces on $D_n$ and $C_2$: these algebras are commutative and hence a trace is just a linear function from them into $F$. 

\section{Preliminaries}

Throughout this paper $F$ will denote a field of characteristic zero and $A$ a unitary associative $F$-algebra with trace $\tr$. We say that $A$ is an algebra with trace if it is endowed with a linear map $\tr\colon A \rightarrow F$ such that for all $a,b \in A$ one has
\[
\tr(ab) = \tr(ba).
\]
In what follows, we shall identify, when it causes no misunderstanding, the element $\alpha \in F$ with $\alpha \cdot 1$, where $1$ is the unit of the algebra. 

Accordingly, one can construct $\Ft$, the free algebra with trace on the countable set $X = \{ x_1, x_2, \ldots \}$, where $\Tr$ is a formal trace. Let $\mathcal{M}$ denote the set of all monomials in the elements of $X$. Then $\Ft$ is the algebra generated by the free algebra $\F$ together with the set of central (commuting) indeterminates $\Tr(M)$, $M \in \mathcal{M}$, subject to the conditions that $\Tr(MN) = \Tr(NM)$,
and $\Tr(\Tr(M)N)=\Tr(M)\Tr(N)$, for all $M$, $N \in \mathcal{M}$. In other words, 
\[
\Ft \cong \F \otimes F[\Tr(M) : M \in \mathcal{M}].
\]
The elements of the free algebra with trace are called trace polynomials. 

A trace polynomial $f(x_1, \ldots, x_n, \Tr) \in \Ft$ is a trace identity for $A$ if, after substituting the variables $x_i$ with arbitrary elements $a_i \in A$ and $\Tr$ with the trace $\tr$, we obtain $0$. We denote by $\Id^{tr}(A)$ the set of trace identities of $A$, which is a trace $T$-ideal ($T^{tr}$-ideal) of the free algebra with trace, i.e., an ideal invariant under all endomorphisms of $\Ft$. 

As in the ordinary case, $\Id^{tr}(A)$ is completely determined by its multilinear polynomials.

\begin{Definition}
The vector space of multilinear elements of the free algebra with trace in the first $n$ variables is called the space of multilinear trace polynomials in $x_1$, \dots, $x_n$ and it is denoted by $MT_n$ ($MT$ comes from \textsl{mixed trace}). Its elements are linear combinations of expressions of the type
\[
\Tr(x_{i_1} \cdots x_{i_a}) \cdots \Tr(x_{j_1} \cdots x_{j_b}) x_{l_1} \cdots x_{l_c}
\]
where $	\left \{ i_1, \ldots, i_a, \ldots, j_1, \ldots, j_b, l_1, \ldots, l_c \right \} = \left \{ 1, \ldots, n \right \} $.
\end{Definition}

The non-negative integer
\[
c_n^{tr}(A) = \dim_F \dfrac{ MT_n}{MT_n \cap \Id^{tr}(A)} 
\]
is called the $n$-th trace codimension of $A$.

A prominent role among the elements of $MT_n$ is played by the so-called pure trace polynomials, i.e., polynomials such that all the variables $x_1$, \dots, $x_n$ appear inside a trace.

\begin{Definition}
The vector space of multilinear pure trace polynomials in $x_1$, \dots, $x_n$ is the space
\[
PT_n = \spa_F 	\left \{ \Tr(x_{i_1} \cdots x_{i_a}) \cdots \Tr(x_{j_1} \ldots x_{j_b}) : \left \{ i_1, \ldots, j_b \right \} = \left \{ 1, \ldots, n \right \} \right \}. 
\]
\end{Definition}

For a permutation $\sigma \in S_n$ we write (in \cite{Berele1996}, Berele uses $\sigma$ instead of $\sigma^{-1}$)
\[
\sigma^{-1} = \left ( i_1 \cdots i_{r_1} \right ) \left ( j_1 \cdots j_{r_2} \right ) \cdots \left ( l_1 \cdots l_{r_t} \right )
\]
as a product of disjoint cycles, including one-cycles and let us assume that $r_1 \geq r_2 \geq \cdots \geq r_t$. In this case we say that $\sigma$ is of cyclic type $\lambda = (r_1, \ldots, r_t)$. 
We then define the pure trace monomial $ptr_{\sigma} \in PT_n$ as
\[
ptr_{\sigma}(x_1, \ldots, x_n) = \Tr \left ( x_{i_1} \cdots x_{i_{r_1}} \right ) \Tr \left ( x_{j_1} \cdots x_{j_{r_2}} \right ) \cdots \Tr \left ( x_{l_1} \cdots x_{l_{r_t}} \right ). 
\]
If $\displaystyle a = \sum_{\sigma \in S_n} \alpha_{\sigma} \sigma \in FS_n$, we also define $\displaystyle ptr_{a}(x_1, \ldots, x_n) = \sum_{\sigma \in S_n} \alpha_{\sigma} ptr_{\sigma}(x_1, \ldots, x_n)$.

It is useful to introduce also the so-called trace monomial $mtr_{\sigma} \in MT_{n-1}$. It is defined so that
\[
ptr_{\sigma}(x_1, \ldots, x_n) = \Tr \left ( mtr_{\sigma}(x_1, \ldots, x_{n-1}) x_n \right ).
\]

Let now $\varphi\colon FS_n \rightarrow PT_n$ be the map defined by $\varphi(a) = ptr_a(x_1, \ldots, x_n)$. Clearly $\varphi$ is a linear isomorphism and so $\dim_F PT_{n} = \dim_F FS_{n} = n!$. 

The following result is well known, and we include its proof for the sake of completeness.

\begin{Remark}
$\dim_F MT_n = (n+1)!$. 
\end{Remark}
\begin{proof}
In order to prove the result we shall construct an isomorphism of vector spaces between $PT_{n+1}$ and $MT_n$. This will complete the proof since $\dim_F PT_{n+1} = (n+1)!$. Let $\varphi\colon PT_{n+1} \rightarrow MT_n$ be the linear map defined by the equality
\[
\varphi 	\left ( \Tr(x_{i_1} \cdots x_{i_a})  \cdots \Tr(x_{j_1} \cdots x_{j_b}) \Tr(x_{l_1} \cdots x_{l_c}) \right ) = \Tr(x_{i_1} \cdots x_{i_a})  \cdots \Tr(x_{j_1} \cdots x_{j_b}) x_{l_1} \cdots x_{l_{c-1}}. 
\]
Here we assume, as we may, that $l_c=n+1$. It is easily seen that $\varphi$ is a linear isomorphism and we are done.
\end{proof}

\section{Matrix algebras with trace}

In this section we study matrix algebras with trace. Let $M_n(F)$ be the algebra of $ n \times n $ matrices over $F$. One can endow such an algebra with the usual trace on matrices, denoted $t_1$, and defined as
\[
t_1(a) = t_1  \begin{pmatrix} a_{11} & \cdots & a_{1n} \\ \vdots &
\ddots & \vdots \\ a_{n1} & \cdots &
a_{nn} \end{pmatrix}  = a_{11} + \cdots + a_{nn} \in F.
\]

We recall that every trace on $M_n(F)$ is proportional to the usual trace $t_1$. The proof of this statement is a well known result of elementary linear algebra, we give it here for the sake of completeness.

\begin{Lemma} \label{traces on matrices}
Let $f\colon M_n(F) \rightarrow F$ be a trace. Then there exists $\alpha \in F$ such that $f = \alpha t_1$.
\end{Lemma}
\begin{proof} 
Let $e_{ij}$'s denote the matrix units. First we shall prove that $f(e_{ij}) = 0$ whenever $i \neq j$. In fact, since $f(ab) = f(ba)$, for all $a$, $b \in M_n(F)$, we get that 
\[
f(e_{ij}) = f(e_{ij} e_{jj}) = f(e_{jj} e_{ij}) = f(0) = 0. 
\]
Moreover $f(e_{jj}) = f(e_{11})$, for all $j = 2, \ldots, n$. Indeed $ f(e_{11}) = f(e_{1j} e_{jj} e_{j1}) = f(e_{j1} e_{1j} e_{jj}) = f(e_{jj})$. For any matrix $a \in M_n(F)$, $a = (a_{ij}) = \sum_{i,j} a_{ij} e_{ij}$, we get that
\[
f(a) = f \Bigl( \sum_{i,j} a_{ij} e_{ij} \Bigr) = \sum_{j= 1}^n a_{jj} f(e_{jj}) = f(e_{11}) t_1(a),
\]
and the proof is complete.
\end{proof}

In what follows we shall use the notation $t_\alpha$ to indicate the trace on $M_n(F)$ such that $t_\alpha = \alpha t_1$. Moreover, $M_n^{t_\alpha}$ will denote the algebra of $n \times n$ matrices endowed with the trace $t_\alpha$. 

In sharp contrast with the above result, there are very many different traces on the algebra $D_n = D_n(F)$ of $n \times n$ diagonal matrices over $F$.  

\begin{Remark}\label{traces_on_Dn}
If $\tr$ is a trace on $D_n$ then there exist scalars $\alpha_1$, \dots, $\alpha_n\in F$ such that for each diagonal matrix $a=\mbox{diag}(a_{11},\ldots, a_{nn})\in D_n$ one has $\tr(a) = \alpha_1 a_{11}+\cdots+ \alpha_n a_{nn}$. 
\end{Remark}

The algebra $D_n\cong F^n$ is commutative, and $D_n\cong F^n$ with component-wise operations. Hence a linear function $\tr \colon D_n\to F$ must be of the form stated in the remark. Clearly for each choice of the scalars $\alpha_i$ one obtains a trace on $D_n$, and we have the statement of the remark.

We shall denote with the symbol $t_{\alpha_1, \ldots, \alpha_n}$ the trace $\tr$ on $D_n$ such that, for all $a = \mbox{diag}(a_{11}, \ldots, a_{nn})$, $ \tr(a) = \alpha_1 a_{11} + \cdots + \alpha_n a_{nn}$. Moreover, $D_n^{t_{\alpha_1, \ldots, \alpha_n}}$ will indicate the algebra $D_n$  endowed with the trace $t_{\alpha_1, \ldots, \alpha_n}$. 

Let $(A, t)$ and $(B, t')$ be two algebras with trace. A homomorphism (isomorphism) of algebras $\varphi: A \rightarrow B$ is said to be a homomorphism (isomorphism) of algebras with trace if $\varphi(t(a)) = t'(\varphi(a))$, for any $a \in A$.

We have the following remark. 

\begin{Remark} \label{isomorphic Dn}
Let $S_n$ be the symmetric group of order $n$ on the set $\{ 1,2, \ldots, n \}$. For all $\sigma \in S_n$, the algebras $D_n^{t_{\alpha_1, \ldots, \alpha_n}}$ and $D_n^{t_{\alpha_{\sigma(1)}, \ldots, \alpha_{\sigma(n)}}}$ are isomorphic, as algebras with trace.
\end{Remark}
\begin{proof}
We need only to observe that the linear map $\varphi \colon D_n^{t_{\alpha_1, \ldots, \alpha_n}} \rightarrow D_n^{t_{\alpha_{\sigma(1)}, \ldots, \alpha_{\sigma(n)}}}$, defined by $ \varphi(e_{ii}) = e_{\sigma(i) \sigma(i)}$, for all $i = 1, \ldots, n$, is an isomorphism of algebras with trace.
\end{proof}

Recall that a trace function $\tr$ on an algebra $A$ is said to be degenerate if there exists a non-zero element $a \in A$ such that
\[
\tr(ab) = 0
\]
for every $b\in A$. This means that the bilinear form $f(x,y) = \tr(xy)$ is degenerate on $A$. 

In the following lemma we describe the non-degenerate traces on $D_n$. 

\begin{Lemma}\label{non_degenerate_traces_on_Dn}
Let $D_n^{t_{\alpha_1, \ldots, \alpha_n}}$ be the algebra of $n \times n$ matrices endowed with the trace  
$t_{\alpha_1, \ldots, \alpha_n}$. Such a trace is non-degenerate if and only if all the scalars $\alpha_i$ are non-zero. 
\end{Lemma}
\begin{proof}
Let $t_{\alpha_1, \ldots, \alpha_n}$ be non-degenerate and suppose that there exists $i$ such that $\alpha_i = 0$. Consider the matrix unit $e_{ii}$. It is easy to see that we reach a contradiction since, for any element $\mbox{diag}(a_{11},\ldots, a_{nn}) \in D_n$, we get
$$
t_{\alpha_1, \ldots, \alpha_n}(e_{ii} \mbox{diag}(a_{11},\ldots, a_{nn})) = t_{\alpha_1, \ldots, \alpha_n}(e_{ii} a_{ii} ) = \alpha_i a_{ii} = 0. 
$$ 

In order to prove the opposite direction, let us assume that all the scalars $\alpha_i$ are non-zero. Suppose, by contradiction, that the trace $t_{\alpha_1, \ldots, \alpha_n}$ is degenerate. Hence there exist a non-zero element $a = \mbox{diag}(a_{11},\ldots, a_{nn}) \in D_n$ such that $t_{\alpha_1, \ldots, \alpha_n}(ab) = 0$, for any $b \in D_n$. In particular, let $b = e_{ii}$, for $i=1, \ldots, n$. We have that 
$$
t_{\alpha_1, \ldots, \alpha_n}( a e_{ii}) = t_{\alpha_1, \ldots, \alpha_n}(\mbox{diag}(a_{11},\ldots, a_{nn}) e_{ii}) = t_{\alpha_1, \ldots, \alpha_n}(a_{ii} e_{ii} ) = \alpha_i a_{ii} = 0.
$$ 
Since $\alpha_i \neq 0$, for all $i = 1, \ldots, n$, we get that $a_{ii} = 0$ and so $a = 0$, a contradiction.
\end{proof}

\section{The algebras $D_2^{t_{\alpha, \beta}}$}

In this section we deal with the algebra $D_2$ of $2 \times 2$ diagonal matrices over the field $F$. In accordance with the results of Section $3$, we can define on $D_2$, up to isomorphism, only the following trace functions:
\begin{enumerate}
\item[1.] $t_{\alpha, 0}$, for any $\alpha \in F $, 
\vspace{0.1 cm}
\item[2.] $t_{\alpha, \alpha}$, for any non-zero $\alpha \in F $, 
\vspace{0.1 cm}
\item[3.] $t_{\alpha, \beta }$, for any distinct non-zero $\alpha, \beta \in F $.
\vspace{0.1 cm}
\end{enumerate}

In the first part of this section our goal is to find the generators of the trace $T$-ideals of the identities of the algebra $D_2$ endowed with all possible trace functions. 

Let us start with the case of $D_2^{t_{\alpha,0}}$. Recall that, if $\alpha = 0$, then $D_2^{t_{0,0}}$ is the algebra $D_2$ with zero trace. So $\I^{tr}(D_2^{t_{0,0}})$ is generated by the commutator $[x_1, x_2]$ and $\Tr(x)$ and $c_n^{tr}(D_2^{t_{0,0}}) = c_n(D_2^{t_{0,0}}) =1$.  

For $\alpha \neq 0$, we have the following result.

\begin{Theorem} \label{identities and codimensions of D2 t alpha 0}
Let $\alpha \in F \setminus \{ 0 \}$. The trace $T$-ideal $\Id^{tr}(D_2^{t_{\alpha, 0}})$ is generated, as a trace $T$-ideal, by the polynomials:
\begin{itemize}
\item[•] $ f_1 = [x_1,x_2]$,
\vspace{0.2 cm} 
\item[•] $f_2 = \Tr(x_1) \Tr(x_2) - \alpha \Tr(x_1 x_2)$.
\vspace{0.1 cm}
\end{itemize}
Moreover 
\[
c_n^{tr}(D_2^{t_{\alpha, 0}}) =  2^{n}.
\]
\end{Theorem}
\begin{proof}
It is clear that $T = \langle f_1, f_2 \rangle_{T^{tr}} \subseteq \Id^{tr}(D_2^{t_{\alpha, 0}})$. 

We need to prove the opposite inclusion. Let $f \in MT_n$ be a multilinear trace polynomial of degree $n$. It is clear that $f$ can be written (mod $T$) as a linear combination of the polynomials
\begin{equation} \label{non identities of D2 t alpha 0}
\Tr(x_{i_1} \cdots x_{i_k}) x_{j_1} \cdots x_{j_{n-k}}, 
\end{equation}
where $ \left \{ i_1, \ldots, i_k, j_1, \ldots, j_{n-k}  \right \} = \left \{ 1, \ldots, n \right \}$,  $i_1 < \cdots < i_k$ and $j_1 < \cdots < j_{n-k}$.

Our goal is to show that the polynomials in \eqref{non identities of D2 t alpha 0} are linearly independent modulo $ \Id^{tr}(D_2^{t_{\alpha, 0}})$. To this end, let $g = g(x_1, \ldots, x_n, \Tr)$ be a linear combination of the above polynomials which is a trace identity:
\[
g(x_1, \ldots, x_n, \Tr) = \sum_{I,J} a_{I,J} \Tr(x_{i_1} \cdots x_{i_k}) x_{j_1} \cdots x_{j_{n-k}},
\]
where $I = \{ x_{i_1}, \ldots, x_{i_k} \}$, $J = \{ x_{j_1}, \ldots, x_{j_{n-k}} \}$, and $i_1 < \cdots < i_k$, $j_1 < \cdots < j_{n-k}$. 

We claim that $g$ is actually the zero polynomial. Suppose that, for some fixed $I = \{ x_{i_1}, \ldots, x_{i_k} \}$ and $J = \{ x_{j_1}, \ldots, x_{j_{n-k}} \}$, one has that $a_{I,J} \neq 0$. We consider the following evaluation:
\[
x_{i_1} = \cdots = x_{i_k} = e_{11}, \ \ \ \ \ \ \ \
x_{j_1} = \cdots = x_{j_{n-k}} = e_{22}, \ \ \ \ \ \ \ \
\Tr = t_{\alpha, 0}. 
\]
It follows that $g(e_{11}, \ldots, e_{11}, e_{22}, \ldots, e_{22}, t_{\alpha,0}) = a_{I,J} \alpha e_{22} = 0$. Hence $a_{I,J} = 0$, a contradiction. The claim is proved and so
\[
\Id^{tr}(D_2^{t_{\alpha, 0}}) = T.
\]

Finally, in order to compute the $n$-th trace codimension sequence of our algebra, we have only to count how many elements in \eqref{non identities of D2 t alpha 0} there are. Fixed $k$, there are exactly $\binom{n}{k}$ elements of the type  $\Tr(x_{i_1} \cdots x_{i_k}) x_{j_1} \cdots x_{j_{n-k}}$, $i_1 < \cdots < i_k$, $j_1 < \cdots < j_{n-k}$. Hence the polynomials in \eqref{non identities of D2 t alpha 0} are exactly $ \sum_{k=0}^n \binom{n}{k} = 2^n$ and the proof is complete.
\end{proof}

Now, we consider $D_2^{t_{\alpha,\alpha}}$. Recall that, for any $\begin{pmatrix} 
a & 0 \\
0 & b
\end{pmatrix} \in D_2$, we have that $
t_{\alpha, \alpha} \begin{pmatrix} 
a & 0 \\
0 & b
\end{pmatrix} = \alpha (a +  b).$ 

\begin{Theorem} \label{identities of D2 talpha alpha}
Let $\alpha \in F \setminus \{ 0 \}$. The trace $T$-ideal $\Id^{tr}(D_2^{t_{\alpha, \alpha}})$ is generated, as a trace $T$-ideal, by the polynomials:
\begin{itemize}
\item[•] $ f_1 = [x_1,x_2]$,
\vspace{0.3 cm}
\item[•] $ f_3 = \alpha^2 x_1x_2 + \alpha^2 x_2 x_1 +  \Tr(x_1)\Tr(x_2) - \alpha \Tr(x_1)x_2 - \alpha \Tr(x_2)x_1 - \alpha \Tr(x_1 x_2) $.
\vspace{0.1 cm}
\end{itemize}
Moreover 
\[
c_n^{tr}(D_2^{t_{\alpha, \alpha}})= 2^n.
\]
\end{Theorem}
\begin{proof}
In case $\alpha=1$, Berele (\cite[Theorem $2.1$]{Berele1996}) proved that $\Id^{tr}(D_2^{t_{\alpha, \alpha}}) = \langle f_1, f_3 \rangle_{T^{tr}}$. The proof when $\alpha\ne 1$ follows word by word that one given by Berele in \cite{Berele1996}. 

In order to find the trace codimensions, we remark that the trace polynomials
\begin{equation} \label{non identities of D2 with trace}
\Tr(x_{i_1} \cdots x_{i_k}) x_{j_1} \cdots x_{j_{n-k}}, \ \ \ \ \ \ \ \ \ \ \left \{ i_1, \ldots, i_k, j_1, \ldots, j_{n-k}  \right \} = \left \{ 1, \ldots, n \right \}, \ \  i_1 < \cdots < i_k, \ \ j_1 < \cdots < j_{n-k}, 
\end{equation}
form a basis of $MT_n \pmod{ MT_n \cap \Id^{tr}(D_2^{t_{\alpha, \alpha}})}$. Hence, their number, which is the $n$-th trace codimension sequence of $D_2^{t_{\alpha, \alpha}}$, is $\sum_{k=0}^n \binom{n}{k} = 2^n$ and the proof is complete.
\end{proof}

\begin{Remark}
\label{differentvarieties}
Here we observe a curious fact. It follows from Theorems \ref{identities and codimensions of D2 t alpha 0} and \ref{identities of D2 talpha alpha} that the relatively free algebras in the varieties  of algebras with trace generated by $D_2^{t_{\alpha,0}}$ and by $D_2^{t_{\alpha,\alpha}}$ are quite similar. In fact the multilinear components of degree $n$ in these two relatively free algebras are isomorphic. But this is an isomorphism of vector spaces which cannot be extended to an isomorphism of the corresponding algebras. It can be easily seen that neither of these two varieties is a subvariety of the other as soon as $\alpha\ne 0$. 
\begin{enumerate}
\item
The trace identity $f_2=Tr(x_1)Tr(x_2)-\alpha Tr(x_1x_2)$ does not hold for the algebra $D_2^{t_{\alpha,\alpha}}$. One evaluates it on the ``generic" diagonal matrices $d_1=diag(a_1,b_1)$ and $d_2=diag(a_2,b_2)$ and gets $\alpha^2(a_1b_2 + a_2b_1)$ which does not vanish on $D_2^{t_{\alpha,\alpha}}$. 
\item
Likewise $D_2^{t_{\alpha,0}}$ does not satisfy the trace identity $f_3$. Once again substituting $x_1$ and $x_2$ in $f_3$ by the generic matrices $d_1$ and $d_2$ we get a diagonal matrix with 0 at position $(1,1)$ and a non-zero entry $\alpha^2(2b_1b_2  -a_1b_2-a_2b_1)$ at position $(2,2)$. 
\end{enumerate}
This question is addressed in a more general form in Lemmas~\ref{D2 delta 0 no T equ}, \ref{D2 gamma gamma no T equ}, and \ref{D2 alfa beta no T equ}.
\end{Remark}

Finally we consider the trace algebra $D_2^{t_{\alpha,\beta}}$. Recall that, for any $\begin{pmatrix} 
a & 0 \\
0 & b
\end{pmatrix} \in D_2$, we have that
$
t_{\alpha, \beta} \begin{pmatrix} 
a & 0 \\
0 & b
\end{pmatrix} = \alpha a + \beta b.
$ 

\begin{Theorem} \label{identities and codimensions of D2 talpha beta}
Let $\alpha$, $\beta \in F \setminus \{ 0 \}$, $\alpha \neq \beta$. As a trace $T$-ideal, $\Id^{tr}(D_2^{t_{\alpha, \beta}})$ is generated by the polynomials:
\begin{itemize}
\vspace{0.1 cm}
\item[•] $ f_1 = [x_1,x_2]$,
\vspace{0.3 cm} 
\item[•] $f_4 = -x_1 \Tr(x_2) \Tr(x_3) + (\alpha + \beta) x_1 \Tr(x_2 x_3)  + x_3 \Tr(x_1) \Tr(x_2) - (\alpha + \beta) x_3 \Tr(x_1 x_2) - \Tr(x_1) \Tr(x_2 x_3) + \Tr(x_3)\Tr(x_1 x_2)$,
\vspace{0.3 cm} 
\item[•] $f_5 = \Tr(x_1) \Tr(x_2) \Tr(x_3) - (\alpha \beta^2 + \alpha^2 \beta) x_1 x_2 x_3 + \alpha \beta x_1 x_2 \Tr(x_3) + \alpha \beta x_1 x_3 \Tr(x_2) + \alpha \beta x_2 x_3 \Tr(x_1) - (\alpha + \beta) x_1 \Tr(x_2) \Tr(x_3) + (\alpha^2 + \alpha \beta + \beta^2) x_1 \Tr(x_2 x_3) - \alpha \beta x_2 \Tr(x_1 x_3) - \alpha \beta x_3 \Tr(x_1 x_2) + \alpha \beta \Tr(x_1 x_2 x_3) - (\alpha + \beta)  \Tr(x_1) \Tr(x_2 x_3)$.
\end{itemize}
Moreover 
\[
c_n^{tr}(D_2^{t_{\alpha, \beta}}) =  2^{n+1}-n-1.
\]
\end{Theorem}
\begin{proof}
Write $I = \langle f_1, f_4, f_5 \rangle_{T^{tr}}$. An immediate (but tedious) verification shows that $ I \subseteq \Id^{tr}(D_2^{t_{\alpha, \beta}})$. In order to obtain the opposite inclusion, first we shall prove that the polynomials 
\begin{equation} \label{non identities of D2 alfa beta 1}
x_{i_1} \cdots x_{i_{k}} \Tr(x_{h_1} \cdots x_{h_{n-k}}), \ \ \ \ \
x_{i_1} \cdots x_{i_{k}} \Tr(x_{j_1} \cdots x_{j_{s-1}}) \Tr(x_{j_s}), 
\end{equation}
where $  i_1 < \cdots < i_k$, $ h_1 < \cdots < h_{n-k}$ and $ j_1 < \cdots < j_{s-1} < j_s$, span $ MT_n $, modulo $ MT_n \cap I $, for every $n \geq 1$. 

In order to achieve this goal we shall use an induction. Let $f \in MT_n$ be a multilinear trace polynomial of degree $n$. Hence it is a linear combination of polynomials of the type
\[
 x_{i_1} \cdots x_{i_a} \Tr(x_{j_1} \cdots x_{j_b}) \cdots \Tr(x_{l_1} \cdots x_{l_c})
\]
where $	\left \{ i_1, \ldots, i_a, j_1, \ldots, j_b, \ldots, l_1, \ldots, l_c \right \} = \left \{ 1, \ldots, n \right \} $.

Because of the identity $f_5 \equiv 0$, we can kill all products of three traces (and more than three traces). So we may consider only monomials with either no trace, or with one, or with two traces. Clearly the identity $f_1$ implies that we can assume all of these monomials ordered,  outside and also inside each trace.

In the case of monomials with two traces, now we want to show how to reduce one of these traces to be of a monomial of length $1$ (that is a variable). To this end, in $f_4$ take $x_1$ as a letter, and $x_2$, $x_3$ as monomials. The last term is ``undesirable'', and it is written as a combination of either one trace, or two traces where one of these is a trace of a letter (that is $x_1$). The only problem is the first term of $f_4$. But it has a letter outside the traces. Since the total degree of the monomials inside traces in the first summand of $f_4$ will be less than the initial one, we can apply the induction.

Suppose now we have a linear combination of monomials where we have either no traces (at most one of these), or just one trace (and the variables outside the trace are ordered, as well as those inside the trace), or two traces. In the latter case we may assume that the variables outside the trace are ordered, and that these in the first trace are ordered (increasing) as well. And moreover the second trace is of a variable, not monomial. In this case we are concerned with the monomials having two traces. We use once again $f_4$, in fact the last two summands in it, to exchange
variables between the two traces. We get then monomials with one trace or with two traces but of lower degree inside the traces, and as above continue by induction.

In conclusion, we can suppose that in the case of two traces, the variables are ordered in the following way:
\[
x_{i_1} \cdots x_{i_a} \Tr(x_{j_1}\cdots x_{j_b}) \Tr(x_j)
\]
where $i_1<\cdots<i_a$ and $j_1<\cdots<j_b<j$. 

We next show that the polynomials in \eqref{non identities of D2 alfa beta 1} are linearly independent modulo $ \Id^{tr}(D_2^{t_{\alpha, \beta}})$.

Let us take generic diagonal matrices $X_i=(a_i, b_i)$, that is we consider $a_i$ and $b_i$ as commuting independent variables. If the monomials we consider are not linearly independent there will be a non-trivial linear combination among them which vanishes. Form such a linear combination and evaluate it on the above defined generic diagonal matrices $X_i$. We order the monomials in $a_i$ and $b_i$ obtained in the linear combination, at positions $(1,1)$ and $(2,2)$ of the resulting matrix, as follows. 

In the first coordinate (that is position $(1,1)$ of the matrices) we consider $a_1<\cdots<a_n<b_1<\cdots<b_n$, in the second coordinate $(2,2)$ of the matrices in $D_2$ we take $b_1<\cdots<b_n<a_1<\cdots<a_n$. Then we extend this order lexicographically to all monomials in $F[a_i, b_i]$. In fact these are two orders, one for position $(1,1)$ and another for position $(2,2)$ of the diagonal matrices.

In order to simplify the notation, let us assume that the largest monomial in the first coordinate is $a_1 \cdots a_k b_{k+1} \cdots b_n$. If $k=n$ then there is no trace at all. The case $k=n-1$ is also clear: we have only one trace and it is $\Tr(x_n)$. So take $k\le n-2$. Such a monomial can come from either $M = x_1\cdots x_k \Tr(x_{k+1}\cdots x_n)$ or from $N = x_1\cdots x_k \Tr(x_{k+1}\cdots x_{n-1})\Tr(x_n)$. Clearly in the second coordinate the largest monomial will be $b_1\cdots b_k a_{k+1}\cdots a_n$, and it comes from the above two elements only.

Now suppose that a linear combination of monomials vanishes on the generic matrices $X_i$. Then the largest monomials will cancel and so there will exist scalars $p$,
$q\in F$  such that the largest monomials in $pM+qN$ will cancel. This means, after computing the traces, that $p\beta + q\beta^2 = 0$ and $p\alpha+q\alpha^2=0$. Consider $p$ and $q$ as variables in a $2\times
2$ system. The determinant of the system is $\alpha\beta(\alpha-\beta)$. Since $\alpha$ and $\beta$ are both non-zero and since $\alpha\ne \beta$, we get that $p=q=0$ and we cancel out the largest monomials. 

In conclusion the polynomials in \eqref{non identities of D2 alfa beta 1} are linearly independent modulo $ \Id^{tr}(D_2^{t_{\alpha, \beta}})$. 

Since $ MT_n \cap \Id^{tr}(D_2^{t_{\alpha, \beta}})  \supseteq MT_n \cap I $, they form a basis of $ MT_n$ modulo $MT_n \cap \Id^{tr}(D_2^{t_{\alpha, \beta}}) $ and  $ \Id^{tr}(D_2^{t_{\alpha, \beta}}) = I $.

Finally, in order to compute the codimension sequence of our algebra, we have only to observe the following facts. We have only one monomial with no trace at all, and exactly $n$ monomials where $n-1$ letters are outside the traces (and the remaining one is inside a trace). Then we have $2{n\choose s}$ elements of the type 
\[
x_{i_1}\cdots x_{i_k} \Tr(x_{j_1}\cdots x_{j_s}) \mbox{\ \ \ \ \  or \ \ \ \ \ } x_{i_1}\cdots x_{i_k} \Tr(x_{j_1}\cdots x_{j_{s-1}}) \Tr(x_{j_s}), \qquad k+s=n. 
\]
In conclusion we get that 
$$ 
c_n^{tr}(D_2^{t_{\alpha, \beta}}) = \binom{n}{0} + 	\binom{n}{1} + 2\sum_{s = 2}^{n} \binom{n}{s} = 2^{n+1}-n-1.$$
\end{proof}

Given a variety $\mathcal{V}$ of algebras with trace, the growth of $\mathcal{V}$ is the growth of the sequence of trace codimensions of any algebra $A$ generating $\mathcal{V}$, i.e., $\mathcal{V} = \V^{tr}(A) $. We say that $\mathcal{V}$ has almost polynomial growth if it grows exponentially but any proper subvariety has polynomial growth. 

In the following theorem we prove that the algebras $D_2^{t_{\alpha, \alpha}}$ generate varieties of almost polynomial growth.

\begin{Theorem} \label{D2 alfa APG}
The algebras $D_2^{t_{\alpha, \alpha}}$, $\alpha \in F \setminus \{ 0 \}$, generate varieties of almost polynomial growth.
\end{Theorem}
\begin{proof}
By Theorem \ref{identities of D2 talpha alpha}, the variety generated by $D_2^{t_{\alpha, \alpha}}$ has exponential growth. 

We are left to prove that any proper subvariety of $\V^{tr}(D_2^{t_{\alpha, \alpha}})$ has polynomial growth. Let $\V^{tr}(A) \subsetneq \V^{tr}(D_2^{t_{\alpha, \alpha}})$. Then there exists a multilinear trace polynomial $f$ of degree $n$ which is a trace identity for $A$ but not for $D_2^{t_{\alpha, \alpha}}$. We can write $f$ as
\begin{equation} 
f = \sum_{k=0}^n \sum_{I} \alpha_{k,I,J} \Tr(x_{i_1} \cdots x_{i_k}) x_{j_1} \cdots x_{j_{n-k}} + h
\end{equation}
where $h \in \Id^{tr}(D_2^{t_{\alpha, \alpha}})$, $I = \{ i_1, \ldots, i_k \}$, $J = \{ j_1, \ldots, j_{n-k} \}$, $i_1 < \cdots < i_k$ and $j_1 < \cdots < j_{n-k}$.

Let $M$ be the largest $k$ such that $\alpha_{k,I,J} \neq 0$. There may exist several monomials in $f$ with that same $k$, we choose the one with the least monomial with respect to its trace part $x_{i_1} \cdots x_{i_k}$ (in the usual lexicographical order on the monomials in $x_1$, \dots, $x_n$ induced by $x_1<\cdots<x_n$). 

Now consider a monomial $g$ of degree $n' > n+M$ of the type
\[
g = \Tr(x_{l_1} \cdots x_{l_a}) x_{k_1} \cdots x_{k_{n'-a}}
\] 
with $a > 2M$ and $n'-a > n-M$. We split the monomial $x_{l_1} \cdots x_{l_a}$ inside the trace, in $M$ monomials $y_1 = x_{l_1} \cdots x_{l_{a_1}}$, \dots, $y_M = x_{l_{a_{M-1}+1}} \cdots x_{l_{a}}$, each one with $\lfloor \frac{a}{M} \rfloor$ or $\lceil \frac{a}{M} \rceil$ variables. We also let $y_{M+1} = x_{k_1}$, \dots, $y_{n'-a+M} = x_{k_{n'-a}}$.

Now, because of $f(y_1, \ldots, y_n) \equiv 0$, we can write $g \pmod{\Id^{tr}(A)}$  as a linear combination of monomials having either less than $M$ variables $y_i$ inside the trace, or $M$ variables $y_i$ inside the trace, but at least one of these variables is not among $y_1$, \dots, $y_M$. Passing back to $x_1$, \dots, $x_{n'}$ we see that $g$ is a linear combination of monomials with less than $a$ variables inside the trace. If $a > 2M$ is still satisfied for some of these monomials (for the new value of $a$) we repeat the procedure and so on.

Thus after several such steps we shall write $g$ as a linear combination of monomials with at most $2M$ variables inside the traces. It follows that, for $n$ large enough,
\[
c_n^{tr}(A) \leq \sum_{k = 0}^{2M} \binom{n}{k} \approx bn^{2M}
\]
where $b$ is a constant. Hence $\V^{tr}(A)$ has polynomial growth and the proof is complete.
\end{proof}

With a similar proof we obtain also the following result. 

\begin{Theorem} \label{D2 alfa 0 APG}
The algebras $D_2^{t_{\alpha, 0}}$, $\alpha \in F \setminus \{ 0 \}$, generate varieties of almost polynomial growth.
\end{Theorem}

We conclude this section by proving some results showing that the algebras $D_2^{t_{\alpha, \beta}}$, $D_2^{t_{\gamma, \gamma}}$ and $D_2^{t_{\delta, 0}}$ are not $T^{tr}$-equivalent. Recall that given two algebras with trace $A$ and $B$, $A$ is $T^{tr}$-equivalent to $B$ and we write $A \sim_{T^{tr}} B$, in case $\Id^{tr}(A) = \Id^{tr}(B)$. 

\begin{Lemma} \label{D2 delta 0 no T equ}
Let $\alpha$, $\beta, \gamma, \delta, \epsilon \in F \setminus \{0 \}$, $\alpha \neq \beta$, $\delta \neq \epsilon$. Then
\begin{itemize}
\item[1.]  $\Id^{tr}(D_2^{t_{\delta, 0}}) \not \subset \Id^{tr}(D_2^{t_{\alpha, \beta}})$.
\vspace{0.1 cm}
\item[2.] $ \Id^{tr}(D_2^{t_{\delta, 0}}) \not \subset  \Id^{tr}(D_2^{t_{\gamma,\gamma}})$.
\vspace{0.1 cm}
\item[3.]  $ \Id^{tr}(D_2^{t_{\delta, 0}}) \not \subset  \Id^{tr}(D_2^{t_{\epsilon, 0}})$.
\vspace{0.1 cm}
\end{itemize}
\end{Lemma}
\begin{proof}
Let us consider the polynomial 
\[
f_2 =  \Tr(x_1)\Tr(x_2) - \delta \Tr(x_1 x_2). 
\]
We have seen in Theorem \ref{identities and codimensions of D2 t alpha 0} that $f_2$ is a trace identity of $D_2^{t_{\delta, 0}}$. In order to complete the proof we need only to show that $f_2$ does not vanish on the algebras $D_2^{t_{\alpha, \beta}}$, $D_2^{t_{\gamma, \gamma}}$ and $D_2^{t_{\epsilon, 0}}$. By considering the evaluation $x_1 = e_{11}$ and $x_2 = e_{22}$, we obtain that $f_2(e_{11}, e_{22}, t_{\alpha, \beta}) = \alpha \beta (e_{11}+ e_{22}) \neq 0$ and $f_2(e_{11}, e_{22}, t_{\gamma, \gamma}) =\gamma^2 (e_{11} + e_{22}) \neq 0$. Hence $f_2$ is not a trace identity of $D_2^{t_{\alpha, \beta}}$ and $D_2^{t_{\gamma, \gamma}}$ and we are done in the first two cases. Finally, evaluating $x_1 = x_2 = e_{11}$, we get $f_2(e_{11}, e_{11}, t_{\epsilon, 0}) = \epsilon (\epsilon - \delta) (e_{11} + e_{22}) \neq 0$ and the proof is complete.
\end{proof}

\begin{Lemma} \label{D2 gamma gamma no T equ}
Let $\alpha$, $\beta, \gamma, \delta, \kappa \in F \setminus \{0 \}$, $\alpha \neq \beta$, $\gamma \neq \kappa$. Then
\begin{itemize}
\item[1.]  $\Id^{tr}(D_2^{t_{\gamma, \gamma}}) \not \subset \Id^{tr}(D_2^{t_{\alpha, \beta}})$.
\vspace{0.1 cm}
\item[2.] $ \Id^{tr}(D_2^{t_{\gamma, \gamma}}) \not \subset  \Id^{tr}(D_2^{t_{\kappa,\kappa}})$.
\vspace{0.1 cm}
\item[3.]  $ \Id^{tr}(D_2^{t_{\gamma, \gamma}}) \not \subset  \Id^{tr}(D_2^{t_{\delta, 0}})$.
\vspace{0.1 cm}
\end{itemize}
\end{Lemma}
\begin{proof}
Let us consider the polynomial 
$$
f_3 = \gamma^2 x_1x_2 + \gamma^2 x_2 x_1 +  \Tr(x_1)\Tr(x_2) - \gamma \Tr(x_1)x_2 - \gamma \Tr(x_2)x_1 - \gamma \Tr(x_1 x_2). 
$$
We have seen in Theorem \ref{identities of D2 talpha alpha} that $f_3$ is a trace identity of $D_2^{t_{\gamma, \gamma}}$. By considering the evaluation $x_1 = e_{11}$ and $x_2 = e_{22}$, we obtain that $f_3(e_{11}, e_{22}, t_{\alpha, \beta}) = \beta (\alpha - \gamma) e_{11} + \alpha (\beta - \gamma) e_{22} \neq 0$, $f_3(e_{11}, e_{22}, t_{\kappa, \kappa}) =\kappa(\kappa- \gamma) (e_{11} + e_{22}) \neq 0$ and $f_3(e_{11}, e_{22}, t_{\delta, 0}) = -\gamma \delta e_{22} \neq 0$.  Hence $f_3$ is not a trace identity of $D_2^{t_{\alpha, \beta}}$,  $D_2^{t_{\kappa, \kappa}}$ and $D_2^{t_{\delta, 0}}$ and the proof is complete.
\end{proof}

\begin{Lemma} \label{D2 alfa beta no T equ}
Let $\alpha$, $\beta, \gamma, \delta, \eta, \mu \in F \setminus \{0 \}$, $\alpha \neq \beta$, $\eta \neq \mu$, $\{ \alpha, \beta \} \neq \{ \eta, \mu \}$. Then
\begin{itemize}
\item[1.]  $\Id^{tr}(D_2^{t_{\alpha, \beta}}) \not \subset \Id^{tr}(D_2^{t_{\eta, \mu}})$.
\vspace{0.1 cm}
\item[2.] $\Id^{tr}(D_2^{t_{\alpha, \beta}}) \not \subset  \Id^{tr}(D_2^{t_{\gamma,\gamma}})$.
\vspace{0.1 cm}
\item[3.]  $\Id^{tr}(D_2^{t_{\alpha, \beta}}) \not \subset  \Id^{tr}(D_2^{t_{\delta, 0}})$.
\vspace{0.1 cm}
\end{itemize}
\end{Lemma}
\begin{proof}
By Theorem \ref{identities and codimensions of D2 talpha beta} we know that the polynomials $f_4$ and $f_5$ are trace identities of $D_2^{t_{\alpha, \beta}}$. In order to complete the proof we shall show that such polynomials do not vanish on the algebras $D_2^{t_{\eta, \mu}}$, $D_2^{t_{\gamma, \gamma}}$ and $D_2^{t_{\delta, 0}}$. 
\begin{itemize}
\item[1.] 
We have to consider two different cases. If $\alpha + \beta \neq \eta + \mu$, then $f_4(e_{11}, e_{22}, e_{22}, t_{\eta, \mu}) = \mu (\alpha + \beta - \eta - \mu) e_{11} \neq 0$ and we are done in this case. Now, let us suppose that $\alpha + \beta = \eta + \mu$. In this case, for some $\lambda \in F$, we obtain that $f_5(e_{11}, e_{22}, e_{22}, t_{\eta, \mu}) = \lambda e_{11} + \eta (\beta - \mu) (\alpha - \mu) e_{22}$ is non-zero since the hypothesis $\{ \eta, \mu \} \neq \{ \alpha, \beta \}$ implies that $\beta \neq \mu$ and $\alpha \neq \mu$.
\vspace{0.1 cm}
\item[2.] 
It is the same proof of item $1.$ in which $\eta = \mu = \gamma$. 
\vspace{0.1 cm}
\item[3.] The evaluation $x_1 = x_2 = e_{22}$, $x_3 = e_{11}$ gives $f_5(e_{22}, e_{22}, e_{11}, t_{\delta, 0}) = \alpha \beta \delta e_{22} \neq 0$.
\end{itemize}
\end{proof}

\section{The algebras $C_2^{t_{\alpha, \beta}}$}

In this section we focus our attention on the $F$-algebra
$$
C_2 = 	\left \{ \begin{pmatrix} 
a & b \\
0 & a
\end{pmatrix} : a,b \in F \right \}.
$$

Since $C_2$ is commutative, every trace on $C_2$ is just a linear map $C_2 \rightarrow F$. Hence, if $\tr$ is a trace on $C_2$, then there exist $\alpha, \beta \in F$ such that
$$
\tr \left ( \begin{pmatrix} 
a & b \\
0 & a
\end{pmatrix} \right ) = \alpha a + \beta b.
$$
We denote such a trace  by $t_{\alpha, \beta}$. Moreover,  $C_2^{t_{\alpha, \beta}}$ indicates the algebra $C_2$ endowed with the trace $t_{\alpha, \beta}$.

\begin{Lemma} \label{identities of C2 t alfa 0}
Let $\alpha \in F$. Then $C_2^{t_{\alpha, 0}}$ satisfies the following trace identities of degree $2:$
\begin{itemize}
\item[1.] $[x_1, x_2] \equiv 0$.
\vspace{0.1 cm}
\item[2.] $\Tr(x_1) \Tr(x_2) - \alpha \Tr(x_1 x_2) \equiv 0$.
\vspace{0.1 cm}
\item[3.] $\Tr(x_1) \Tr(x_2) - \alpha \Tr(x_1) x_2 - \alpha \Tr(x_2) x_1 + \alpha^2 x_1 x_2 \equiv 0$.
\end{itemize}
\end{Lemma}
\begin{proof}
The result follows by an immediate verification.
\end{proof}

If $\alpha = 0$ then $C_2^{t_{0,0}}$ is a commutative algebra with zero trace and $c_n^{tr}(C_2^{t_{0,0}}) = 1$, for all $n \geq 1$. In case $\alpha \neq 0$, by putting together Lemma \ref{identities of C2 t alfa 0} and Theorem \ref{identities and codimensions of D2 t alpha 0} we get that $\V^{tr}(C_2^{t_{\alpha,0}}) \subsetneq \V^{tr}(D_2^{t_{\alpha,0}}) $. Hence, by Theorem \ref{D2 alfa 0 APG}, $C_2^{t_{\alpha,0}}$ generates a variety of polynomial growth. 

\begin{Remark} \label{C alpha, beta equivalent to C alpha beta'}
Let $\alpha, \beta, \beta' \in F$ with $\beta, \beta' \neq 0$. The algebras $C_2^{t_{\alpha, \beta}}$ and $C_2^{t_{\alpha, \beta'}}$ are isomorphic, as algebras with trace.
\end{Remark}
\begin{proof}
We need only to observe that the linear map $\varphi \colon C_2^{t_{\alpha, \beta}} \rightarrow C_2^{t_{\alpha, \beta'}}$, defined by 
$$
\varphi \left ( \begin{pmatrix} 
a & b \\
0 & a
\end{pmatrix} \right ) = \begin{pmatrix} 
a & \beta \beta'^{-1} b \\
0 & a
\end{pmatrix},
$$
is an isomorphism of algebras with trace.
\end{proof}

With a straightforward computation we get the following result. 

\begin{Lemma} \label{identity of C alfa}
Let $\alpha \in F$. Then:
\begin{itemize}
\item[1.] $C_2^{t_{\alpha,1}}$ does not satisfy any multilinear trace identity of degree $2$ which is not a consequence
of $[x_1, x_2] \equiv 0$. 
\vspace{0.1 cm}
\item[2.] $C_2^{t_{\alpha,1}}$ satisfies the following trace identity of degree $3:$
$$
f_\alpha = \alpha x_1 x_2 x_3 + \Tr(x_1 x_2) x_3 + \Tr(x_1 x_3) x_2 + \Tr(x_2 x_3) x_1 - \Tr(x_1) x_2 x_3 - \Tr(x_2) x_1 x_3 - \Tr(x_3) x_1 x_2 - \Tr(x_1 x_2 x_3).
$$ 
\end{itemize}
\end{Lemma}

Next we shall prove that, for any $\alpha \in F$, the algebra $C_2^{t_{\alpha,1}}$ generates a variety of exponential growth. 

\begin{Theorem} \label{C alfa has exp growth}
For any $\alpha \in F$, the algebra $C_2^{t_{\alpha,1}}$ generates a variety of exponential growth. 
\end{Theorem}
\begin{proof}
Let us consider the following set of trace monomials of degree $n$:
\begin{equation} \label{monomials}
\Tr(x_{i_1}) \cdots \Tr(x_{i_k}) x_{j_1} \cdots x_{j_{n-k}},
\end{equation}
where $\{ i_1, \ldots, i_k, j_1, \ldots, j_{n-k} \} = \{ 1, 	\ldots, n \}$, $i_1 < \cdots < i_k$, $j_1 < \cdots < j_{n-k}$, $k = 0, \ldots, n$.

The number of elements in \eqref{monomials} is exactly $
\sum_{k = 0}^n \binom{n}{k} = 2^n$. So, in order to prove the theorem we shall show that the monomials in \eqref{monomials} are linearly independent, modulo $\Id^{tr}(C_2^{t_{\alpha,1}})$. To this end, let $g \in \Id^{tr}(C_2^{t_{\alpha,1}})$ be a linear combination of the above elements:
$$
g(x_1, \ldots, x_n) = \sum_{I,J} a_{I,J} \Tr(x_{i_1}) \cdots \Tr(x_{i_k}) x_{j_1} \cdots x_{j_{n-k}},
$$
where $k= 0, \ldots, n$, $I = \{ x_{i_1}, \ldots, x_{i_k} \}$, $J = \{ x_{j_1}, \ldots, x_{j_{n-k}} \}$ and $i_1 < \cdots < i_k, \ j_1 < \cdots < j_{n-k}$. 

We claim that $g$ is actually the zero polynomial. Let $k$ be the largest integer such that $\alpha_{I,J} \neq 0$, with fixed $I = \{ x_{i_1}, \ldots, x_{i_k} \}$ and $J = \{ x_{j_1}, \ldots, x_{j_{n-k}} \}$. By making the evaluation $x_{i_1} = \cdots = x_{i_k} = e_{12}$ and $x_{j_1} = \cdots = x_{j_{n-k}} = e_{11} + e_{22}$, we get $g = \alpha_{I,J} (e_{11} + e_{22}) + \gamma e_{12} = 0$. This implies $ \alpha_{I,J} = 0$, a contradiction.  
\end{proof}

We conclude this section with the following results comparing trace $T$-ideals.

\begin{Lemma} \label{C alfa no T equ C beta}
Let $\alpha, \beta \in F$ be two distinct elements. Then $\Id^{tr}(C_2^{t_{\alpha,1}}) \not \subset \Id^{tr}(C_2^{t_{\beta,1}}) $.
\end{Lemma}
\begin{proof}
Let us consider the polynomial 
$$
f_\alpha = \alpha x_1 x_2 x_3 + \Tr(x_1 x_2) x_3 + \Tr(x_1 x_3) x_2 + \Tr(x_2 x_3) x_1 - \Tr(x_1) x_2 x_3 - \Tr(x_2) x_1 x_3 - \Tr(x_3) x_1 x_2 - \Tr(x_1 x_2 x_3).
$$ 
We have seen in Lemma \ref{identity of C alfa} that $f_\alpha$ is a trace identity of $C_2^{t_{\alpha,1}}$. In order to complete the proof we need only show that such a polynomial does not vanish on $C_2^{t_{\beta,1}}$. By considering the evaluation $x_1 = x_2 = x_3 = e_{11} + e_{22} \in C_2^{t_{\beta,1}} $, we get 
\[
f_\alpha(e_{11} + e_{22}, e_{11} + e_{22}, e_{11} + e_{22}, t_{\beta,1}) = (\alpha - \beta) (e_{11} + e_{22}). 
\]
Since $\alpha \neq \beta $, $f_\alpha$ does not vanish on $C_2^{t_{\beta,1}}$ and we are done.
\end{proof}

\begin{Lemma} \label{D2 alpha 0, D2 alfa alfa no T equ C beta}
Let $\alpha, \beta, \gamma, \delta \in F \setminus \{0 \}$, $\epsilon \in F$, $\alpha \neq \beta$. Then 
\begin{itemize}
\item[1.] $\Id^{tr}(D_2^{t_{\delta,0}}) \not \subset \Id^{tr}(C_2^{t_{\epsilon,1}}) $,
\vspace{0.1 cm}
\item[2.] $\Id^{tr}(D_2^{t_{\gamma,\gamma}}) \not \subset \Id^{tr}(C_2^{t_{\epsilon,1}}) $,
\vspace{0.1 cm}
\item[3.] $\Id^{tr}(D_2^{t_{\alpha,\beta}}) \not \subset \Id^{tr}(C_2^{t_{\epsilon,1}}) $.
\end{itemize}
\end{Lemma}
\begin{proof}
By Theorems \ref{identities and codimensions of D2 t alpha 0} and \ref{identities of D2 talpha alpha}, we know that the algebras $D_2^{t_{\delta,0}}$ and $D_2^{t_{\gamma,\gamma}}$ satisfy trace identities of degree $2$ which are not a consequence of $[x_1, x_2] \equiv 0$. This does not happen for the algebra $C_2^{t_{\epsilon,1}}$ (see the first item of Lemma \ref{identity of C alfa}) and so the proof of the first two items is complete. 

In order to prove the last item, let us consider the polynomial $f_5$ of Theorem \ref{identities and codimensions of D2 talpha beta}, which is a trace identity of $D_2^{t_{\alpha, \beta}}$.  Such a polynomial does not vanish on $C_2^{t_{\epsilon,1}}$. In fact, by considering the evaluation $x_1 = x_2 = x_3 = e_{12} \in C_2^{t_{\epsilon,1}} $, we get 
\[
f_5(e_{12}, e_{12}, e_{12}, t_{\epsilon,1}) = e_{11} + e_{22} - (\alpha + \beta)e_{12} \neq 0. 
\]
\end{proof}

\begin{Lemma} \label{C alfa no equ D2 beta gamma}
Let $\alpha, \beta, \gamma \in F$, $\alpha \neq 0$. Then $\Id^{tr}(C_2^{t_{\gamma,1}}) \not \subset \Id^{tr}(D_2^{t_{\alpha, \beta}}) $.
\end{Lemma}
\begin{proof}
Let us consider the polynomial $f_\gamma$ of Lemma \ref{identity of C alfa}, which is a trace identity of $C_2^{t_{\gamma,1}}$. We shall show that it does not vanish on $D_2^{t_{\alpha, \beta}}$. By considering the evaluation $x_1 = x_2 = e_{11}$ and $x_3 = e_{22}$, we get
\[
f_\gamma(e_{11}, e_{11}, e_{22}, t_{\alpha, \beta}) = \alpha e_{22} - \beta e_{11}. 
\]
Since $\alpha \neq 0$, $f_\gamma$ does not vanish on $D_2^{t_{\alpha, \beta}}$ and we are done.

In particular, in case $\beta = \alpha$ we get that  $\Id^{tr}(C_2^{t_{\gamma,1}}) \not \subset \Id^{tr}(D_2^{t_{\alpha, \alpha}}) $ and in case $\beta = 0$ $\Id^{tr}(C_2^{t_{\gamma,1}}) \not \subset \Id^{tr}(D_2^{t_{\alpha, 0}}) $. 
\end{proof}

\section{Algebras with trace of polynomial growth}

We start this section by describing a version of the Wedderburn-Malcev theorem for finite dimensional algebras with trace.
First we recall some definitions. Let $A$ be a unitary algebra with trace $\tr$. A subset (subalgebra,  ideal) $ S \subseteq A$  is a trace-subset (subalgebra, ideal) of $A$ if it is stable under the trace; in other words for all $ s \in S $, one has $ \tr(s) \in S $. 

\begin{Definition}
Let $A$ be an algebra with trace. $A$ is called a trace-simple algebra if 
\begin{enumerate}
\item[1.] $A^2 \neq 0$,
\vspace{0.1 cm}
\item[2.] $A$ has no non-trivial trace-ideals.
\end{enumerate}
\end{Definition}

\begin{Remark} \label{simple implies trace simple}
Let $A$ be an algebra with trace $\tr$. 
\begin{enumerate}
\item[1.] If $A$ is simple (as an algebra) then $A$ is trace-simple.
\vspace{0.1 cm}
\item[2.] If $I$ is a proper trace-ideal of $A$ then the trace vanishes on $I$.
\end{enumerate}
\end{Remark}
\begin{proof}
The first item is obvious. For the second one, let us suppose that there exists $a\in I$ such that $tr(a)=\alpha\ne 0$. Hence $\alpha \in F$ is invertible. Moreover, since $I$ is a trace-ideal, it contains $\alpha$ and so we would have $I=A$, a contradiction.

Notice that the second item of the remark also holds for one-sided ideals. 
\end{proof}

In the following result we give a version of the Wedderburn--Malcev theorem for finite dimensional algebras with trace. 

\begin{Theorem}\label{WM}
Let $A$ be a finite dimensional unitary algebra with trace $\tr$ over an algebraically closed field $F$ of characteristic $0$. Then there exists a semisimple trace-subalgebra $B$ such that
\[
A=B+J(A) = B_1 \oplus \cdots \oplus B_k + J(A)
\]
where $J = J(A)$ is the Jacobson radical of $A$ and $B_1$, \dots, $B_k$ are simple algebras.
\end{Theorem}
\begin{proof}
By the Wedderburn--Malcev theorem for the ordinary case (see for example \cite[Theorem $3.4.3$]{GiambrunoZaicev2005book}), we can write $A$ as a direct sum of vector spaces
\[
A = B + J = B_1 \oplus \cdots \oplus B_k + J
\]
where $B$ is a maximal semisimple subalgebra of $A$, $J= J(A)$ is the Jacobson radical of $A$, and $B_i$ are simple algebras, $i = 1$, \dots, $k$. By the Theorems of Wedderburn and Wedderburn--Artin on simple and semisimple algebras (see for instance \cite[Theorems 1.4.4, 2.1.6]{Herstein1968book}), and since $F$ is algebraically closed, we have that
\[
B = B_1 \oplus \cdots \oplus B_k = M_{n_1}(F) \oplus \cdots \oplus M_{n_k}(F).
\]
Here $M_{n_i}(F)$ is the simple algebra of $n_i \times n_i$ matrices, $i = 1$, \dots, $k$. Clearly $B$ is a trace-subalgebra since $1_A \in B$. Moreover, by considering the restriction of the trace $\tr$ on $B$ it is easy to see that there exist $\alpha_i \in F$ such that
\[
tr(a_1, \ldots, a_k) = \sum_{i = 1}^{k} t_{\alpha_i}(a_i)
\]
where $a_i \in M_{n_i}(F)$, $t_{\alpha_i} = \alpha_i t_1^i$, and $t_1^i$ is the ordinary trace on the matrix algebra $M_{n_i}(F)$. 
\end{proof}

In order to prove the main result of this paper we need the following lemmas.

\begin{Lemma} \label{C alfa in F+J}
Let $A = B + J$ be a finite dimensional algebra with trace $\tr$. If there exists $j \in J$ such that $\tr(j) \neq 0$ then $C_2^{t_{\alpha,1}} \in \V^{tr}(A)$, for some $\alpha \in F$.
\end{Lemma}
\begin{proof}
Let us consider the trace subalgebra $B'$ of $A$ generated by $1$, $j$ over $F$ and let $I$ be the ideal of $B'$ generated by $j^n$, where $n$ is the least integer such that $\tr(j^n) = \tr(j^{n+1}) = \cdots = 0$. Then the quotient algebra $\bar{B} = B'/I$ is an algebra with trace $t$ defined as $t(a+I) = \tr(a)$, for any $a \in B'$. Obviously $\bar{B} = \mbox{span} \{ \bar{1} = 1+I, \bar{j} = j+I, \ldots,  \bar{j}^{n-1} = j^{n-1}+I\}$. Let $\alpha = \tr(1)$ and $\beta = \tr({j}^{n-1}) \neq 0$.

We claim that $C_2^{t_{\alpha, \beta}} \in \V^{tr}(\bar{B})$. Let $\varphi\colon C_2^{t_{\alpha, \beta}} \to \bar{B} $ be the linear map defined by $\varphi(e_{11}+e_{22}) = \bar{1}$ and $\varphi(e_{12}) = \bar{j}^{n-1}$. It is easy to check that $\varphi$ is an injective homomorphism of algebras with trace. Hence $C_2^{t_{\alpha, \beta}}$ is isomorphic to a trace subalgebra of $\bar{B}$ and the claim is proved. 

Since, by Remark \ref{C alpha, beta equivalent to C alpha beta'}, $C_2^{t_{\alpha, \beta}} \cong C_2^{t_{\alpha, 1}}$, it follows that $C_2^{t_{\alpha,1}} \in \V^{tr}(A)$ and the proof is complete. 
\end{proof}

\begin{Lemma} \label{D2 alfa alfa in Mn alfa}
For any $\alpha \in F$, the algebra $D_2^{t_{\alpha, \alpha}}$ belongs to the variety generated by $M_n^{t_\alpha}$.
\end{Lemma}
\begin{proof}
Let us recall that we denote by $M_n^{t_\alpha}$ the algebra of the $n\times n$ matrices endowed with the trace $t_\alpha$; this is the usual trace multiplied by the scalar $\alpha\in F$. 
Since $D_2^{t_{\alpha,\alpha}} \subseteq M_2^{t_\alpha}$ as algebras with (the same) trace it follows that $D_2^{t_{\alpha,\alpha}}$ satisfies all trace identities of $M_2^{t_\alpha}$ (and some additional ones). Therefore $D_2^{t_{\alpha,\alpha}} \in \V^{tr}(M_2^{t_\alpha})$. In order to complete the proof we need just to show that $M_2^{t_{\alpha}} \in \V^{tr}(M_n^{t_\alpha})$. To this end, let $f \in \I^{tr}(M_n^{t_\alpha})$ be a multilinear trace identity of degree $m$ and suppose, by contradiction, that
there exists elementary matrices $e_{i_1 j_1}$, \dots, $e_{i_m j_m}$ in $M_2^{t_\alpha}$ such that $f(e_{i_1 j_1}, \ldots, e_{i_m j_m}) = \sum \alpha_{i, j} e_{i j} \neq 0$. Notice that, if we denote by $e_{ij}'$ the elementary matrices in $M_n^{t_\alpha}$, then $f(e_{i_1 j_1}', \ldots, e_{i_m j_m}') = \sum \alpha_{i, j} e_{i j} + \sum_{i=3}^n \beta_{ii} e_{ii} \neq 0$, a contradiction.
\end{proof}

In order to prove the main result of this paper we have to consider also the algebra $UT_2$ of $ 2 \times 2$ upper-triangular matrices endowed with zero trace. In the following theorem we collect some results concerning such an algebra. 

\begin{Theorem} \label{UT2}
Let $UT_2$ be the algebra of $ 2 \times 2$ upper-triangular matrices endowed with zero trace.
\begin{itemize}
\item[1.] The trace $T$-ideal $\I^{tr}(UT_2)$ is generated by $[x_1, x_2] [x_3, x_4]$ and $\Tr(x)$.
\vspace{0.1 cm}
\item[2.] $UT_2$ generates a variety of almost polynomial growth.
\vspace{0.1 cm}
\item[3.] $ \I^{tr}(UT_2) \nsubseteq \I^{tr}(A)$, where $A \in \{ D_2^{t_{\alpha, \beta}}, D_2^{t_{\gamma,\gamma}}, D_2^{t_{\delta,0}}, C_2^{t_{\epsilon,1}} \}$, $\alpha$, $\beta$, $\gamma$, $\delta \in F \setminus \{ 0 \}$, $\alpha \ne \beta$, $\epsilon \in F$.
\vspace{0.1 cm}
\item[4.] $ \I^{tr}(A) \nsubseteq \I^{tr}(UT_2)$, where $A \in \{ D_2^{t_{\alpha, \beta}}, D_2^{t_{\gamma,\gamma}}, D_2^{t_{\delta,0}}, C_2^{t_{\epsilon,1}} \}$, $\alpha$, $\beta$, $\gamma$, $\delta \in F \setminus \{ 0 \}$, $\alpha \ne \beta$, $\epsilon \in F$.
\vspace{0.1 cm}
\end{itemize}
\end{Theorem}
\begin{proof}
The first two items follows directly from the ordinary case (see, for instance \cite[Chapter 4 and 7]{GiambrunoZaicev2005book}).

For the item (3) it is sufficient to observe that $\Tr(x) \equiv 0$ is a trace-identity of $UT_2$ but such a polynomial does not vanish on $A$, for any $A \in \{ D_2^{t_{\alpha, \beta}}, D_2^{t_{\gamma,\gamma}}, D_2^{t_{\delta,0}}, C_2^{t_{\epsilon,1}} \}$. Finally, since the algebras $D_2^{t_{\alpha, \beta}}, D_2^{t_{\gamma,\gamma}}, D_2^{t_{\delta,0}}, C_2^{t_{\epsilon,1}}$ are commutative and $UT_2$ is not, we get item (4), and the proof is complete. 
\end{proof}

Now we are in a position to prove the following theorem characterizing the varieties of unitary algebras with trace which are generated by finite dimensional algebras, and have  polynomial growth of their codimensions.

\begin{Theorem} \label{characterization}
Let $A$ be a finite dimensional unitary algebra with trace $\tr$ over a field $F$ of characteristic zero. Then the sequence $c_n^{tr}(A)$,  $n=1$, 2, \dots, is polynomially bounded if and only if $ D_2^{t_{\alpha, \beta}}$,  $D_2^{t_{\gamma,\gamma}}$, $D_2^{t_{\delta,0}}$, $C_2^{t_{\epsilon,1}}, UT_2 \notin \V^{tr}(A)$, for any choice of $\alpha$, $\beta$, $\gamma$, $\delta \in F \setminus \{ 0 \}$, $\alpha \ne \beta$, $\epsilon \in F$.
\end{Theorem}
\begin{proof}
By Theorems \ref{identities and codimensions of D2 t alpha 0}, \ref{identities of D2 talpha alpha}, \ref{identities and codimensions of D2 talpha beta}, \ref{C alfa has exp growth}, \ref{UT2}, the algebras $D_2^{t_{\alpha, \beta}}$, $D_2^{t_{\gamma,\gamma}}$, $D_2^{t_{\delta,0}}$, $C_2^{t_{\epsilon,1}}$ and $UT_2$ generate varieties of exponential growth. Hence, if $c_n^{tr}(A)$ is polynomially bounded, then $  D_2^{t_{\alpha, \beta}}$, $D_2^{t_{\gamma,\gamma}}$, $  D_2^{t_{\delta,0}}$, $C_2^{t_{\epsilon,1}}, UT_2 \notin \V^{tr}(A)$, for any $\alpha$, $\beta$, $\gamma$, $\delta \in F \setminus \{ 0 \}$, $\alpha \ne\beta$, $\epsilon \in F$.

Conversely suppose that $ D_2^{t_{\alpha, \beta}}$, $D_2^{t_{\gamma,\gamma}}$, $ D_2^{t_{\delta, 0}}$, $C_2^{t_{\epsilon,1}}, UT_2 \notin \V^{tr}(A)$, for any $\alpha$, $\beta$, $\gamma$, $ \delta \in F \setminus \{ 0 \}$, $\alpha \ne\beta$, $\epsilon \in F$. Since we are dealing with codimensions, and these do not change under extensions of the base field, we may assume that the field $F$ is algebraically closed. By Theorem \ref{WM}, we get that 
\[
A = M_{n_1}(F) \oplus \cdots \oplus M_{n_k}(F) + J, \ \ k \geq 1,
\]
and there exist constants $\alpha_i$ such that, for $a_i \in M_{n_i}(F)$, we have
\[
tr(a_1, \ldots, a_k) = \sum_{i = 1}^{k} t_{\alpha_i}(a_i).
\]
Since $D_2^{t_{\gamma,\gamma}} \notin \V^{tr}(A)$, for any $\gamma \in F \setminus \{ 0 \}$, and since, by Lemma \ref{D2 alfa alfa in Mn alfa}, we have that, for $n\ge 2$, $D_2^{t_{\gamma,\gamma}} \in \V^{tr}(M^{t_\gamma}_n) \subseteq \V^{tr}(A) $, we get that $n_i = 1$, for every $i = 1$, \dots, $k$. Hence 
\[
A=A_1\oplus \cdots\oplus A_k + J
\]
where for every $i=1$, \dots, $k$, $A_i\cong F$ and the trace on it is $t_{\alpha_i}$.

Since, for any $\alpha \in F$, $C_2^{t_{\alpha,1}} \not \in \V^{tr}(A)$, by Lemma \ref{C alfa in F+J} we must have that the trace vanishes on $J$.

Now, if for any $i = 1$, \dots, $k$, the trace on $A_i$ is zero, since $UT_2 \not \in \V^{tr}(A)$, then, for any $i \neq j$, we must have $A_i J A_j = 0$. Hence, for $n \geq 1$, $c_n^{tr}(A) = c_n(A)$ is polynomially bounded (see, for instance \cite[Chapter 7]{GiambrunoZaicev2005book}) and we are done in this case.

Hence, we may assume that there exists $i$ such that the trace on $A_i$ is $t_{\alpha_i}$, with $\alpha_i \neq 0$.

Let $F_\alpha$ denote the field $F$ endowed with the trace $t_\alpha$. We claim that $F_\alpha \oplus F_\beta$ is isomorphic to $D_2^{t_{\alpha, \beta}}$ if $\alpha \neq \beta$ (notice that $\beta$ could be zero) and to $D_2^{t_{\alpha, \alpha}}$ otherwise. Here we shall denote by $t$ the trace map on $F_{\alpha} \oplus F_{\beta} $ defined as $t((a,b)) = t_{\alpha}(a) + t_{\beta}(b)$, for all $(a,b) \in F_\alpha \oplus F_\beta$. In order to prove the claim, let us consider the linear map $\varphi\colon D_2 \rightarrow F_\alpha \oplus F_\beta $ such that
\[
\varphi   \begin{pmatrix} 1 & 0 \\ 0  &
0 \end{pmatrix}  = (1,0) \ \ \ \ \ \ \ \ \ \mbox{and} \ \ \ \ \ \ \ \ \ \varphi   \begin{pmatrix} 0 & 0 \\ 0  &
1 \end{pmatrix}  = (0,1).
\]
It is easily seen that $\varphi$ is an isomorphism of algebras. 

Now, if $\alpha \neq \beta$, we have that
\[
\varphi \left ( t_{\alpha, \beta}   \begin{pmatrix} 1 & 0 \\ 0  &
0 \end{pmatrix} \right ) = \varphi(\alpha)  = (\alpha, \alpha) = t(1,0) = t \left ( \varphi   \begin{pmatrix} 1 & 0 \\ 0  &
0 \end{pmatrix}  \right ),
\]
\[
\varphi \left ( t_{\alpha, \beta}   \begin{pmatrix} 0 & 0 \\ 0  &
1 \end{pmatrix}  \right ) = \varphi(\beta)  = (\beta, \beta) = t(0,1) = t \left ( \varphi   \begin{pmatrix} 0 & 0 \\ 0  &
1 \end{pmatrix}  \right ),
\]
and so $\varphi$ is an isomorphism of algebras with trace between $D_2^{t_{\alpha, \beta}}$ and $F_\alpha \oplus F_\beta$. In the same way, if $\alpha = \beta$, we get a trace isomorphism between $D_2^{t_{\alpha,\alpha}}$ and $F_\alpha \oplus F_\alpha$. 

Hence, since $D_2^{t_{\alpha, \beta}}, D_2^{t_{\gamma,\gamma}}$, $ D_2^{t_{\delta,0}} \notin \V^{tr}(A)$, it follows that 
\[
A = B + J
\]
where $B \cong F$ and for all $a = b+j \in A$, $tr(a) = tr(b+j) = \alpha b$, with $\alpha \neq 0$. 

In order to complete the proof we need to show that $B + J$ has polynomially bounded trace codimensions. 

Notice that the following polynomials are trace identities of $B + J$:
\begin{enumerate}
\item[1.] $\alpha \Tr(x_1 x_2) - \Tr(x_1) \Tr(x_2) \equiv 0$,
\vspace{0.2 cm}
\item[2.] $ 	\left (  \Tr(x_1) - \alpha x_1 \right )\cdots  \left ( \Tr(x_{q+1}) - \alpha x_{q+1} \right )\equiv 0$, where $J^q \neq 0$ and $J^{q+1} = 0$.
\end{enumerate}
Modulo the first trace identity, every multilinear trace polynomial in the variables $x_1$, \dots, $x_n$ is a linear combination of expressions of the type
\[
\Tr(x_{i_1}) \cdots \Tr(x_{i_a}) x_{j_1} \cdots x_{j_b}
\]
where $	\left \{ i_1, \ldots, i_a, j_1, \ldots, j_b \right \} = \left \{ 1, \ldots, n \right \} $ and $i_1 < \cdots < i_a $. The second identity implies that we can suppose $a\le q$. Indeed if $a\ge q+1$ then one can represent the product of $q+1$ traces as a linear combination of elements with fewer traces by expanding the second trace identity. Therefore we may suppose $a\le q$. 
 
We can further reduce the form of the latter polynomials. As we consider algebras with unit we can rewrite each monomial $x_{j_1} \cdots x_{j_b}$ as a linear combination of elements of the type
\[
x_{p_1} \cdots x_{p_c} \left [ x_{l_1}, \ldots , x_{l_{k}} \right ] \cdots \left [ x_{m_1}, \ldots , x_{m_{h}} \right ].
\]
Here the commutators that appear are left normed, that is $[u,v] = uv-vu$, $[u,v,w] = [[u,v],w]$ and so on. Moreover 
\[
\left \{ i_1, \ldots, i_a, p_1, \ldots, p_c, l_1, \ldots, l_k, \ldots, m_1, \ldots, m_h \right \} = \left \{ 1, \ldots, n \right \}.
\]
By applying the Poincar\'e--Birkhoff--Witt theorem, we can assume further that $p_1 < \cdots < p_c $, and that the commutators are ordered. Therefore we can write each trace polynomial as a linear combination of elements of the type
\[
\Tr(x_{i_1}) \cdots \Tr(x_{i_a}) x_{p_1} \cdots x_{p_c}  [ x_{l_1}, \ldots , x_{l_{k_1}}  ] \cdots  [ x_{m_1}, \ldots , x_{m_{k_h}}  ],
\]
where $	\left \{ i_1, \ldots, i_a, p_1, \ldots, p_c, l_1, \ldots, l_{k_1}, \ldots, m_1, \ldots, m_{k_h} \right \} = \left \{ 1, \ldots, n \right \} $ and $i_1 < \cdots < i_a $, $p_1 < \cdots < p_c $, $a \leq q$. The algebra $B$ in the decomposition $A=B+J$ is commutative. Hence each commutator vanishes when evaluated on elements of $B$ only. Thus in order to get a non-zero element we must substitute in a commutator at least one element from $J$. Since $J^{q+1} = 0$, a product of $q+1$ commutators is a trace identity. Therefore we also get the restriction
$ K  := k_1 + \cdots + k_h \leq q$. In this way 
\begin{eqnarray*}
c_n^{tr}(A) &\le & \sum_{a=0}^q \binom{n}{a}  	\left ( \sum_{K = k_1 + \cdots + k_h = 0}^q 		\binom{n-a}{K}  \binom{K}{k_1, \ldots, k_h } k_1! \cdots k_h!   \right ) \\
&=&\sum_{a=0}^q  \dfrac{n(n-1) \cdots (n-a+1)}{a!}\left ( \sum_{K = k_1 + \cdots + k_h=0}^q 	\dfrac{(n-a)!}{(n-a-K)!}   \right ) \\
& \approx & cn^{2q}
\end{eqnarray*}
where $c$ is a constant. Hence $A = B + J$ has polynomial growth and the proof is complete.
\end{proof}

As an immediate consequence, we get the following result.

\begin{Corollary} If $A$ is a finite dimensional unitary algebra with trace, then the sequence  $c^{tr}_n(A)$, $n=1$, 2, \dots, is either  polynomially bounded or grows exponentially.
\end{Corollary}

Now we prove the following corollary.

\begin{Corollary} Let $\alpha$, $\beta \in F \setminus \{ 0 \}$, $\alpha \neq \beta$. Any proper subvariety of $\V^{tr}(D_2^{t_{\alpha, \beta}})$, generated by a finite dimensional algebra with trace, has polynomial growth.
\end{Corollary} 
\begin{proof}
Let $\mathcal{V} = \V^{tr}(A) \subsetneq \V^{tr}(D_2^{t_{\alpha, \beta}})$  where $A$ is a finite dimensional algebra with trace. As a consequence of Lemmas \ref{D2 alfa beta no T equ}, \ref{D2 alpha 0, D2 alfa alfa no T equ C beta} (item $3.$) and Theorem \ref{UT2}, we get that $UT_2$, $D_2^{t_{\alpha', \beta'}}$, $D_2^{t_{\gamma, \gamma}}$, $D_2^{t_{\delta,0}}$, $ C_2^{t_{\epsilon,1}} \not \in \V^{tr}(A)$, for any $\alpha'$, $\beta'$, $\gamma$, $\delta \in F \setminus \{ 0 \}$, $\alpha' \ne\beta'$, $\epsilon \in F$. Hence Theorem \ref{characterization} applies and the proof is complete.
\end{proof}

With a similar approach we obtain the following result. 

\begin{Corollary} For any $\epsilon \in F$, any proper subvariety of $\V^{tr}(C_2^{t_{\epsilon,1}})$, generated by a finite dimensional algebra with trace, has polynomial growth.
\end{Corollary}

According to the previous results, with an abuse of terminology, we may say that $D_2^{t_{\alpha, \beta}}$ and $C_2^{t_{\epsilon,1}}$ generate varieties of almost polynomial growth. As a consequence we state the following corollary.

\begin{Corollary} The algebras $UT_2$, $D_2^{t_{\alpha, \beta}}$, $D_2^{t_{\gamma,\gamma}}$, $D_2^{t_{\delta,0}}$ and $C_2^{t_{\epsilon,1}}$, $\alpha, \beta, \gamma, \delta \in F \setminus \{ 0 \}$, $\alpha \neq \beta$, $\epsilon \in F$, are the only finite dimensional algebras with trace generating varieties of almost polynomial growth.
\end{Corollary}

\textbf{Acknowledgements}

We thank the Referee whose comments were appreciated. Remark~\ref{differentvarieties} was added in order to answer a question raised by the Referee.

\end{document}